\documentclass[a4paper, 12pt]{amsart}
\usepackage[all,cmtip]{xy}
\usepackage{amssymb}
\usepackage{amsfonts}
\usepackage{amsthm}
\usepackage{amsmath}
\usepackage{amscd}
\usepackage{color} 
\usepackage[utf8]{inputenc}
\usepackage{hyperref}
\usepackage{fullpage}
\usepackage{epic}
\usepackage{mathtools}
\usepackage{tikz}
\usetikzlibrary{chains}

\tikzset{node distance=2em, ch/.style={circle,draw,on chain,inner sep=2pt},chj/.style={ch,join},every path/.style={shorten >=4pt,shorten <=4pt},line width=1pt,baseline=-1ex}

\usepackage{graphicx}
\usepackage[all]{xy}
\usepackage{array}
\usepackage{youngtab}

\newcommand{\fg}{{\mathfrak{g}}}
\newcommand{\fh}{{\mathfrak {h}}}

\newcommand{\bbC}{{\mathbb {C}}}

\newcommand{\bbP}{{\mathbb {P}}}

\newcommand{\bbR}{{\mathbb {R}}}

\newcommand{\bbZ}{{\mathbb {Z}}}

\newcommand{\CH}{{\mathcal {H}}}

\newcommand{\CL}{{\mathcal {L}}}

\newcommand{\tr}{{\mathrm {tr}}}

\newcommand{\Hom}{{\mathrm{Hom}}}

\newcommand{\la}{\langle}
\newcommand{\ra}{\rangle}

\newcommand{\be}{\begin{equation}}
\newcommand{\ee}{\end{equation}}
\newcommand{\bt}{\begin{theorem}}
\newcommand{\et}{\end{theorem}}
\newcommand{\bd}{\begin{definition}}
\newcommand{\ed}{\end{definition}}
\newcommand{\bp}{\begin{proposition}}
\newcommand{\ep}{\end{proposition}}
\newcommand{\bl}{\begin{lemma}}
\newcommand{\el}{\end{lemma}}
\newcommand{\bco}{\begin{corollary}}
\newcommand{\eco}{\end{corollary}}
\newcommand{\br}{\begin{remark}}
\newcommand{\er}{\end{remark}}
\newcommand{\bex}{\begin{example}}
\newcommand{\eex}{\end{example}}
\newcommand{\ben}{\begin{enumerate}}
\newcommand{\een}{\end{enumerate}}
\newcommand{\bc}{\begin{cases}}
\newcommand{\ec}{\end{cases}}

\newcommand{\bpf}{\begin{proof}}
\newcommand{\epf}{\end{proof}}

\newcommand{\bma}{\begin{bmatrix}}
\newcommand{\ema}{\end{bmatrix}}

\theoremstyle{Theorem}

\theoremstyle{Theorem}

\theoremstyle{Theorem}

\theoremstyle{Definition}

\newtheorem{theorem}{Theorem}[section]

\newtheorem{lemma}[theorem]{Lemma}
\newtheorem{proposition}[theorem]{Proposition}
\newtheorem{corollary}[theorem]{Corollary}

\newtheorem{remark}[theorem]{Remark}

\theoremstyle{definition}
\newtheorem{definition}[theorem]{Definition}

\newtheorem{example}[theorem]{Example}

\begin{document}

\title{Fusion rings revisited}

\author{Jiuzu Hong}
\address{Department of Mathematics, University of North Carolina at Chapel Hill,  Chapel Hill, NC 27599-3250, U.S.A.}
\email{jiuzu@email.unc.edu}
\maketitle
\begin{abstract}
In this note we describe a general elementary procedure to attach a fusion ring to any Kac-Moody algebra of affine type.   In the case of untwisted affine algebras, they are  usual fusion rings in the literature.  In the case of twisted affine algebras, they are exactly the twisted fusion rings defined  by the author in \cite{Ho2} via tracing out diagram automorphisms on conformal blocks for appropriate simply-laced Lie algebras.  We also relate the fusion ring to the modular S-matrix for any Kac-Moody algebra of affine type. 
\end{abstract}
\section{ Kac-Moody algebras of affine types}

Let $A$ be a generalized Cartan matrix of affine type of order $n+1$ (and rank $n$).  We denote by $a_i (i=0, 1, \cdots, n)$  the labelling on the vertices of  Dynkin diagram associated to $A$ (\cite[\S 4.8, Table Aff 1, Table Aff 2]{Ka}).  We denote by $\check{a}_i (i=0,1,\cdots, n)$  the labelling of the vertices of the affine Dynkin diagram associated to the transpose $A^t$ of $A$, which is obtained rom the Dynkin diagram of $A$ by reversing all arrows and keeping the same enumeration of vertices.

 Let $\fg(A)$ denote the Kac-Moody algebra associated to $A$ with Cartan subalgebra $\fh$, roots $\alpha_i (i=0,1,\cdots n)$ and coroots $\check{\alpha}_i (i=0,1,\cdots, n)$.  Note that $A=(\la \alpha_i, \check{\alpha}_j  \ra)$,where $\la \cdot,\cdot \ra$ is the natural pairing between $\fh$ and its dual $\fh^*$.   Let  $e_i, f_i \, (i=0,1,\cdots,  n)$ denote the Chevalley generators of $\fg(A)$.  The following relations hold:
\[  [e_i,f_j]=\delta_{ij} \check{\alpha}_i,  [h, e_i]=\alpha_i(h)e_i , \text{ and } [h,f_i]=-\alpha_i(h) f_i, \] 
for any $i,j=0,\cdots, n$ and $h\in \fh$.    Let $\fg'(A)$ denote the derived subalgebra $[\fg(A),\fg(A)]$ of $\fg(A)$. Note that $\fg'(A)$ is generated by $e_i,f_i (i=0,\cdots, n)$.    Fix an element $d\in \fh$ such that 
\[ \la \alpha_i,d \ra=0 \text{ for } i=1,\cdots, \quad  \la \alpha_0,d \ra=1  .\]
Then $\fg(A)=\fg'(A)+ \bbC d$, and $\fh$ is spanned by $\check{\alpha}_i (i=0,\cdots, n)$ and $d$.

Put
\[ K=\sum_{i=0}^n  \check{a}_i  \check{\alpha}_i   .\]
 Then $K$ is the central element of $\fg(A)$.

Let $\Lambda_i (i=0,\cdots, n)$ denote the fundamental weights of $\fg(A)$. 
The space $\fh^*$ is spanned by $\alpha_i, i=0,\cdots, n$ and $\Lambda_0$.  There is a special imaginary root $\delta\in \fh^*$ given by $\delta=\sum_{i=0}^n  a_i \alpha_i . $
Let $(\cdot |\cdot )$ denote the symmetric normalized bilinear form on $\fh$ such that:
\[\begin{cases}
(\check{\alpha}_i|  \check{\alpha}_j  )=  a_j \check{a}_j a_{ij}  \quad   (i,j=0,\cdots, n)\\
(\check{\alpha}_i| d)=0  \quad    i=1,\cdots, n  \\
(\check{\alpha}_0  |d )=a_0;  \quad    (d|d)=0.
\end{cases}
\]
  
  Let $\kappa: \fh\to  \fh^*$  be the induced isomorphism from $(\cdot|\cdot)$.  Then 
  \[ \kappa (\check{\alpha}_i)=  \frac{a_i }{\check{a}_i}  \alpha_i ;   \kappa(K)=\delta,    \kappa(d)=a_0 \Lambda_0 .  \]

We denote by  $\mathring{\fg}$ the subalgebra of $\fg(A)$ generated by $e_i,f_i\,(i=1,\cdots, n)$.   
 Let $\mathring{\fh}$ be the span of $\check{\alpha}_i \,(i=1,\cdots, n)$.  Then $\mathring{\fg}$ is a simple Lie algebra with the Cartan subalgebra $\mathring{\fh}$.  Let $\check{h}$ denote the dual Coxeter number of $\fg(A)$.   We denote by $\mathring{A}$ the Cartan matrix of $\mathring{\fg}$.

We will use $X_N^{(r)}$ to denote the  generalized Cartan matrix of affine type. When $r=1$, $X_N^{(1)}$ is untwisted; otherwise $X_N^{(r)}$ is twisted.

  The following is the table for the associated $\mathring{A}$ and $\check{h}$ when $A$ is of untwisted affine type.  
\begin{equation}
\label{Table_Dual_Coxeter_1}
 \begin{tabular}{|c  | c | c |c |c |c | c| c| c|c |c|c|c|c|c |c ||} 
 \hline
$A$ &   $A_n^{(1)}$  & $B_n^{(1)}$ &  $C_n^{(1)}$ &  $D_n^{(1)}$  &   $E_6^{(1)}$  &   $E_7^{(1)}$   &  $E_8^{(1)} $  & $ F_4^{(1)} $  & $ G_2^{(1)} $     \\ [0.5ex] 
 \hline
$\mathring{A}$ & $A_n$ &  $B_n$ & $C_n$  & $D_n$&  $E_6$ & $E_7$  &  $E_8$ & $F_4$ & $G_2$   \\ 
 \hline
$ \check{h} $ & n+1 &  2n-1 & n+1 &  2n-2  & 12 & 18 & 30 & 9 & 4       \\ [1ex] 
 \hline
\end{tabular} .
\end{equation}
In the above table,  $n\geq 1$ when $A=A_n^{(1)}$, $n\geq 3$ when $A=B_n^{(1)}$,  $n\geq 2$ when $A=C_n^{(1)}$, and $n\geq 4$ when $A=D_n^{(1)}$.  

The following is the table for the associated $\mathring{A}$ and $\check{h} $ when $A$ is of twisted affine type.  
\begin{equation}
\label{Table_Dual_Coxeter_2}
 \begin{tabular}{|c  | c | c |c |c |c | c| c| c|c |c|c|c|c|c |c ||} 
 \hline
$A$   & $A_{2n}^{(2)} $  &  $A_{2n-1}^{(2)}$  & $D_{n+1}^{(2)}$  & $ E_6^{(2)}$  &  $D_4^{(3)} $    \\ [0.5ex] 
 \hline
$\mathring{A}$  &  $C_n$  &   $C_{n}$  &  $B_n$  &  $F_4$ &  $G_2$  \\ 
 \hline
$ \check{h} $ & 2n+1 & 2n & 2n & 12 & 6       \\ [1ex] 
 \hline
\end{tabular}.
\end{equation}
In the above table, $n\geq 2$ when $A=A_{2n}^{(2)}$, $n\geq 3$ when $A=A_{2n-1}^{(2)}$, and $n\geq 3$ when $A=D_{n+1}^{(2)}$.

\section{Fusion rings associated to Kac-Moody algebras of affine types}

 Let $\mathring{P}\subset \mathring{\fh}^*$ (resp. $\mathring{Q}\subset \mathring{\fh}^*$) be the weight lattice  (resp. root lattice) associated to $\mathring{\fg}$.    Let $\check{ \mathring{P}  }$ denote the dual lattice of $\mathring{Q}$, i.e.  
 \[  \check{\mathring{P}} = \{ \check{\mu}\in \mathring{\fh} \,|\,    \la \lambda, \check{\mu}   \ra \in \bbZ ,  \text{ for any } \lambda\in \mathring{Q}   \} . \]
Similarly let $\check{\mathring{Q}}$ denote the dual lattice of $\mathring{P}$.   
Let $\mathring{P}^+$ (resp. $\check{\mathring{P}}^{+}$) be the set of dominant weights (resp. dominant coweights) of $\mathring{\fg}$.   Let $\mathring{\Phi}$ (resp. $\mathring{\Phi}^+$) denote the set of roots (resp. positive roots ) of $\mathring{\fg}$.

  Let $\mathring{\rho}$ (resp. $\check{\mathring{\rho}}$) denote the summation of fundamental weights (resp. coweights) of $\fg$.
Recall the symmetric bilinear form $(\cdot| \cdot)$ on $\fh$.  It restricts to the subspace $\mathring{\fh}$, and it induces the isomorphism $\kappa:  \mathring{\fh}\simeq \mathring{\fh}^*$.  

Put
\begin{equation} 
\label{Lattice_M}
 M =\begin{cases}    
\kappa(\check{\mathring{Q}}) \quad  \text{ if } A = X_N^{(1)}  \text{ or } A=A_{2n}^{(2)}\\
\mathring{Q}  \quad  \text{ otherwise}
\end{cases}.
 \end{equation}
Let $W$ (resp. $\mathring{W}$) denote the Weyl group of $\fg(A)$ (resp. $\mathring{\fg}$).       Then $W$ is naturally isomorphic to $\mathring{W}\ltimes M$ (cf.\cite[Proposition 6.5]{Ka}).

Let $\theta$ be the highest root of $\mathring{\fg}$.  Set 
\begin{equation}
\label{Dual_Theta}
 \check{\theta}=\begin{cases}
\kappa^{-1}(\theta)   \quad \text{ if } A = X_N^{(1)}  \text{ or } A=A_{2n}^{(2)}\\
\text{highest coroot of } \mathring{\fg}    \quad  \text{otherwise}.
\end{cases}   \end{equation}
In fact when $ A = X_N^{(1)} $, $\check{\theta}$ is the highest short coroot of $\mathring{\fg}$; when $A=A_{2n}^{(2)}$,  $\check{\theta}$ is twice of the highest short coroot of $\mathring{\fg}$.

Given any $k\in \mathbb{N}$. We put 
 \begin{equation}
 \label{Weight_Level}
 P_{k}=\{   \lambda\in \mathring{P}^+   \,  | \,   \la  \lambda  , \check{\theta}  \ra     \leq k    \}.     \end{equation}
We also put
\[ \check{P}_k=\{  \check{\lambda}\in \check{\mathring{P}}^+    \,|\,  \la \check{\lambda}, \theta   \ra\leq  k     \}  ,\]
(we only need it when $A=X_N^{(r)}$ with $r>1$ and $A\neq A_{2n}^{(2)}$).

Let $\mathring{T}$ denote the torus $\Hom(\mathring{P}, \bbC^\times )$.    
 We define a finite subset $\Sigma_k$ of $\mathring{T}$. 
When $A=X_N^{(1)}$ or $A=A_{2n}^{(2)}$, we set
\begin{equation} 
\label{Finite_Set_torus}
 \Sigma_k=\{ e^{\frac{2\pi i}{k+\check{h}}(\mathring{\rho}+\lambda|\cdot )  } \in \mathring{T}   \,|\,    \lambda\in P_k   \}  .\end{equation}
  When $A=X_N^{(r)}$ with $r>1$ and $A\neq A_{2n}^{(2)}$,  we set 
\[  \Sigma_k= \{e^{\frac{2\pi i}{k+\check{h}}\la\check{\mathring \rho}+\check{\lambda}, \cdot \ra  }\in \mathring{T}   \,|\,    \check{\lambda}\in \check{P}_k  \}   . \]

Let $R_k(A)$ denote the commutative ring of $\bbC$-valued functions on the finite set $\Sigma_{k}$.   We introduce an involution $*$ on the set $P_{k}$  by sending $\lambda$ to $-w_0(\lambda)$ where $w_0$ is the longest element in $\mathring{W}$.  

 For any $w\in \mathring{W}$,  let $\ell(w)$ be the length of $w$.  
  For any $\lambda\in P_{k}$ and $t\in \Sigma_k$, put 
\be \label{J_function}  J_{\lambda}(t) :=\sum_{w\in \mathring{W}} (-1)^{\ell(w)} w(\lambda+\mathring{\rho})(t)  ,\ee
and  $\mathring{\chi}_\lambda(  t ):=\frac{J_{\lambda} (t)  }{J_{0}(t)  } $. 

Let $\mathring{V}_\lambda$ denote the irreducible representation of $\mathring{\fg}$ of highest weight $\lambda$.   By Weyl character formula $\mathring{\chi}_\lambda(t)$ is the trace of $t$ on the representation $\mathring{V}_\lambda$.
Therefore we get a collection of well-defined functions $\{ \mathring{\chi}_\lambda\,|\,  \lambda\in P_k\}$ on $\Sigma_k$.   


Set $\Delta(t)=|J_{0}(t) |^2$. 
We introduce an Hermitian form $(\cdot, \cdot)$ on $R_k(A)$ as follows, for any $f,g\in R_k(A)$ 
\begin{equation}
\label{Inner_Product}
 (f, g):= \frac{1}{| \mathring{P}/(k+\check{h})M |}  \sum_{t\in \Sigma_k}   f(t)\overline{g(t)} \Delta(t).   
 \end{equation}

\begin{theorem}
\label{Orthnormal_Basis_Thm}
\ben
\item The Hermitian form $(\cdot, \cdot)$ on $R_k(A)$ is an inner product. 
\item The set of functions $\{ \mathring{\chi}_\lambda \,|\, \lambda\in P_k\}$ is an orthonormal  basis of the ring $R_k(A)$ with respect to the inner product $(\cdot, \cdot)$.     
\een
\end{theorem}


As a ring of functions on a finite set,   $R_k(A)$ is semismiple.   For any $\lambda,\mu\in P_k$, there exist unique coefficients $c_{\lambda\mu}^\nu$ such that
\[\mathring{\chi}_\lambda\cdot  \mathring{\chi}_\mu=\sum_{\nu\in P_k}   c_{\lambda\mu}^\nu   \mathring{\chi}_\nu,\]
where $c_{\lambda\mu}^\nu\in \bbC$.   
Observe that $\mathring{\chi}_{\lambda^*}(t)=\overline{\mathring{\chi}_\lambda(t)}$ for any $\lambda\in P_k$ and $t\in \Sigma_k$.  Then an immediate consequence of Theorem \ref{Orthnormal_Basis_Thm} is that the coefficients $c_{\lambda\mu}^\nu$ can be computed by the following formula
 \begin{equation}
 \label{Fusion_Coefficient}
  c^\nu_{\lambda\mu} =\frac{1}{   | \mathring{P} /(k+\check{h} ) M   | } \sum_{t\in \Sigma_k }   \mathring{\chi}_\lambda(  t )\mathring{\chi}_\mu(t  )\mathring{\chi}_{\nu^*}(  t )  \Delta( t )  .
  \end{equation}

We call the ring $R_k(A)$ together with the basis $\{ \mathring{\chi}_\lambda \,|\, \lambda\in P_k\}$  the fusion ring associated to the Kac-Moody algebra $\fg(A)$, and we call
 $c_{\lambda\mu}^\nu$  the fusion coefficients of $R_k(A)$.    
 
 The following theorem asserts the integrability of fusion coefficients, and the stabilization  when the level is sufficient large.
\bt 
\label{Integrality_Thm}
For any $\lambda,\mu,\nu\in P_k$, we have
\ben
\item The coefficients $c^{\nu}_{\lambda\mu}  \in \bbZ$. 
\item   If $\la   \lambda+\mu+\nu, \check{\theta}   \ra   \leq  2k$,  then  
\[  c^\nu_{\lambda\mu}=\dim \Hom_{\mathring{\fg}}(\mathring{V}_\nu, \mathring{V}_\lambda\otimes \mathring{V}_\mu).\]
\een
\et

The following theorem is  an incarnation of \cite[Theorem 1.6]{Ho2} in terms of fusion rings. 
\begin{theorem}
\label{Twining_fusion}
The fusion rings $R_{2k+1}(A_{2n}^{(2)})$ and $R_k(C_n^{(1)})$ are isomorphic. 
\end{theorem}

Theorem \ref{Orthnormal_Basis_Thm}, Theorem \ref{Integrality_Thm} and Theorem \ref{Twining_fusion} are basically the consequence of Verlinde formula for dimension of conformal blocks (cf.\cite{Be}),  and the Verlinde formula for the trace of diagram automorphism on conformal blocks \cite{Ho2} by the author.  In the rest of note we will explain uniformly how they are deduced.  

When $A=X_N^{(1)}$,  the theory of fusion rings and conformal blocks  can be dated back to \cite{Fa,Be,TUY,V}.  The coefficients $c_{\lambda\mu}^\nu$ are always nonnegative, since they can be interpreted as the dimension of conformal blocks.

  When $A=X_N^{(r)}$ with $r>1$, the corresponding theory of fusion rings and its relation with usual conformal blocks is essentially explained  in \cite{Ho2},  although the role of twisted affine Lie algebras was not clearly clarified there.  
   We expect that the fusion coefficients $c_{\lambda\mu}^\nu$ are all nonnegative. For example by Theorem \ref{Twining_fusion} the fusion coefficients for $A_{2n}^{(2)}$ are always nonnegative when the level is odd.   We expect that there is an appropriate theory of twisted conformal blocks whose dimensions would correspond to the fusion coefficients for twisted affine Lie algebras. We also hope that there will be a theory of fusion category for twisted affine Lie algebra in the flavor of \cite{Fi}.


\section{Twining formula and Verlinde formula}
\subsection{Twining formula for tensor invariant space}

For each $\fg(A)$ with $A=X_N^{(r)}$, we attach a unique pair $(\dot{\fg},\sigma)$ where $\dot{\fg}$ is a simple Lie algebra and $\sigma$ is a diagram automorphism on $\dot \fg$ (possibly trivial) of order $r$.   Let $\mathring{A}$ (resp. $\dot A$) denote the Cartan matrix of $\mathring \fg$ (resp. $\dot \fg$).  The Cartan matrix $\dot A$ is attached as follows. 
 If $A$ is of untwisted affine type, then we take $\dot{A}=\mathring{A}$;  if $A$ is of twisted affine type, we have the following table:
\begin{equation}
\label{Orbit_Lie_Table}
 \begin{tabular}{|c  | c | c |c |c |c | c| c| c|c |c|c|c|c|c |c |} 
 \hline
$ A$   & $A_{2n}^{(2)} $  &  $A_{2n-1}^{(2)}$  & $D_{n+1}^{(2)}$  & $ E_6^{(2)}$  &  $D_4^{(3)} $    \\ [0.5ex] 
 \hline
$\mathring{A}$  &  $C_n$  &   $C_{n}$  &  $B_n$  &  $F_4$ &  $G_2$  \\ 
 \hline
$ \dot{A} $ &  $A_{2n}$ &  $ D_{n+1} $ & $A_{2n-1}$ & $E_6$ & $ D_4 $      \\ [1ex] 
 \hline
\end{tabular}.
\end{equation}

  The simple Lie algebra $\mathring{\fg}$ is the orbit Lie algebra of the pair $(\dot{\fg},\sigma)$ in the sense of \cite[Section 2.1]{Ho2}.

Let $\mathring{I}$ (resp. $\dot{I}$) denote the set of vertices of Dynkin diagram of  $\mathring{\fg}$ (resp. $\dot{\fg}$).  We will denote by $\dot P$ (resp. $\dot Q$) the weight lattice (resp. root lattice) of $\dot \fg$, and we denote by $\dot P^+$ the set of dominant weights. Let $\{ \dot \omega_j\,|\, j\in \dot I\}$ denote the fundamental weights of $\dot \fg$.  
 There is a bijection $p:  \mathring{I}\simeq \dot{I}/\sigma$, where $\dot{I}/\sigma$ denotes the set of $\sigma$-orbits on $\dot{I}$. Moreover
there is a bijection $\iota:   \mathring{P}\simeq  \dot{P}^\sigma $ and a projection $\check{\iota}:   \check{\dot{Q} }\to \check{\mathring{Q} } $ such that 
\ben
\item   $\iota(\mathring{\omega}_i )=\sum_{ j\in  p(i)   }  \dot{  \omega }_i   $ for any $i\in \mathring{I}$.  
\item  $ \la  \lambda,  \check{\iota}(\check{\beta})  \ra=\la \iota( \lambda), \beta  \ra$, for any $\lambda\in \mathring{P}$ and $\check{\beta}\in \check{\dot Q}$.
\een

For any $\lambda\in \mathring{P}^+$, let 
$\dot V_{\lambda}$ denote the irreducible representation of  $\dot{\fg}$ of highest weight $\iota(\lambda)\in \dot{P}^+$.   There is a unique operator $\sigma: \dot{V}_{\lambda}\to \dot{V}_{\lambda}$ such that 
\begin{enumerate}
\item $\sigma(x\cdot v)=\sigma(x)\cdot \sigma(v)$ for any $x\in \dot \fg$ and $v\in \dot V$,
\item $\sigma({v}_\lambda)={v}_\lambda$ where ${v}_\lambda$ is the highest weight vector of ${\dot V}_{\lambda}$.  
\end{enumerate}

For any tuple $\vec{\lambda}=(\lambda_1,\cdots, \lambda_m)$ of elements in  $\mathring{P}^+$, let $\dot{V}_{\vec{\lambda}}^{\dot \fg}$ (resp. $\mathring{V}_{\vec{\lambda}}^{\mathring{\fg}}$ ) denote the tensor invariant space 
\[  ( {\dot V}_{\lambda_1}\otimes \cdots \otimes {\dot V}_{\lambda_m})^{\dot \fg}   \text{ (  resp. }  ( {\mathring V}_{\lambda_1}\otimes \cdots \otimes {\mathring V}_{\lambda_m})^{\dot \fg}   ). \]

From each operator $\sigma$ on $\dot V_{\lambda_i}$, there exists an operator $\sigma$ (we still use the same notation) acting diagonally on ${\dot V}_{\vec{\lambda}} ^{\dot \fg} $.   The following twining formula was proved in \cite[Theorem 1.1]{HS}. 

\bt
\label{Jantzen_Thm}
We have the equality 
  $\tr(\sigma| \dot{V}^{\dot \fg}_{\vec{\lambda}})=\dim \mathring{V}_{\vec{\lambda}}^{\mathring \fg} $, where $\tr(\sigma |  \dot{V}^{\dot \fg}_{\vec{\lambda}}   )$ is the trace of $\sigma$ on $ \dot{V}^{\dot \fg}_{\vec{\lambda}}$.
\et

The twining formula for the weight spaces between representations of $\dot{\fg}$ and $\mathring{\fg}$ was first discovered by Jantzen (cf. \cite{Ho1,Ja}).  Theorem \ref{Jantzen_Thm} can actually be deduced from Jantzen formula (cf.\cite[Section 5.1]{Ho2}). 

\subsection{Verlinde formula}

Let $\bbC((t))$ be the field of Laurent series. 
Let $\hat{\CL}(\dot  \fg)$ be the associated affine Lie algebra $\fg((t))\oplus \bbC K \oplus \bbC d $ (cf.\cite[\S 7]{Ka}).  The diagram automorphism $\sigma$ still acts on $\hat{\CL}(\dot \fg)$.
Note that if $A$ is of untwisted affine type,  then $\fg(A)\simeq \hat{\CL}(\dot \fg)$. If $A$ is of twisted affine type, then the Kac-Moody algebra $\fg(A)$ is the orbit Lie algebra of  $(\hat{\CL}(\dot \fg),\sigma)$, and in fact there is also a twining formula for the weight spaces between representations of $\hat{\CL}(\dot \fg)$ and $\fg(A)$, which is due to Fuchs-Schellekens-Schweigert\cite{FSS}.

Let $\check{\dot \theta}$ denote the highest short coroot of $\dot \fg$.  
Let ${\dot P}_k$ denote the subset of $\dot P^+$ consisting of $\lambda\in \dot P^+$ such that $\la  \lambda ,   \check{ \dot \theta} \ra \leq  k$.  For any $\lambda\in  \dot P_k$,  let  ${\dot \CH}_{\lambda+k\dot{\Lambda}_0 }$ denote the irreducible integrable representation of $\hat{\CL}(\dot \fg)$ of highest weight $\lambda+k {\dot \Lambda}_0$, where $\dot \Lambda_0$ is the fundamental weight of $\hat{\CL}(\dot \fg)$ corresponding to the vertex $0$ in the extended Dynkin diagram of $\dot \fg$.

\bl
The isomorphism $\iota:  \mathring{P}\simeq  \dot{P}^\sigma$ restricts to the bijection $\iota:  P_k \simeq (\dot{P}_k)^\sigma$. 
\el
\bpf
 It is enough to show that for any $\lambda\in P_k$,  $\iota(\lambda)\in \dot{P}_k$.   
  By  \cite[Lemma 2.1]{Ho2},  $ \check{\iota}(\check{\dot \theta})=\check{ \theta}$.  For any $\lambda\in P_k$,  we have 
\[  \la \iota( \lambda), \check{\dot \theta} \ra=\la\lambda,  \check{\iota}(  \check{\dot \theta}  )   \ra = \la\lambda,  \check{ \theta}  \ra\leq k.  \]
Hence $\iota(\lambda)\in \dot{P}_k$.
\epf

Given an  $m$-pointed smooth projective curve $(C, \vec{p})$ over $\bbC$, on each point $p_i$ we associate a dominant weight $\lambda_i\in \dot P_k$ and an irreducible integral representation ${\dot \CH}_{\lambda_i+k\dot{\Lambda}}$
an affine Lie algebra $\hat{\CL}(\dot \fg)$ depending on the point $p_i$.  
Let ${\dot \fg}(C\backslash \vec{p})$ be the space of ${\dot \fg}$-valued regular functions on $C\backslash \vec{p}$.    The space ${\dot \fg}(C\backslash \vec{p})$ is naturally a Lie algebra induced from $\dot \fg$.  The Lie algebra  ${\dot \fg}(C\backslash \vec{p})$ acts on  
\[ {\dot  \CH}_{\vec{\lambda}}:=\dot{\CH}_{\lambda_1+k\dot{\Lambda}_0  }\otimes \cdots  \otimes \dot{\CH}_{\lambda_m+k \dot{\Lambda}_0 }  \]
naturally.  The space   $\dot{V}_{ \vec{\lambda}}(C,\vec{p})$ of conformal blocks associated to $\vec{p}$ and $\vec{\lambda}$  is defined as the coinvariant space of ${\dot  \CH}_{\vec{\lambda}}$ with respect to the action of ${\dot \fg}(C\backslash \vec{p})$ : 
$$ \dot{V}_{ \vec{\lambda}}(C,\vec{p}):=\dot{\CH}_{\vec{\lambda}}/ {\dot \fg}(C\backslash \vec{p})\dot{\CH}_{\vec{\lambda}}.$$
The trace  $\tr(\sigma | \dot{V}_{ \vec{\lambda}}(C,\vec{p}))$  is independent of the choice of $\vec{p}$ (cf.\cite[Section 3]{Ho2}).  For convenience we write $\tr(\sigma | \dot{V}_{ \vec{\lambda}}(C))$ by ignoring $\vec{p}$.

\bl
\label{Conformal_Orth}
We attach $\lambda, \mu\in \dot{P}_k$ to two distinct points on the projective line $\bbP^1$. Then  we have
\[\tr(\sigma| \dot{V}_{\lambda,\mu}(\bbP^1) ) =  \delta_{\lambda, \mu^*}  . \]
\el
\bpf
cf.\cite[Lemma 3.9]{Ho2}.
\epf

Let $\check{\dot h}$ denote the dual Coxeter number of $\hat{\CL}(\dot \fg)$.  By comparing Table (\ref{Table_Dual_Coxeter_1}), (\ref{Table_Dual_Coxeter_2}) and (\ref{Orbit_Lie_Table}), we can see that $\check{\dot h}=\check{h}$.

\bl
\label{M_Untwisted}
If $A=X_N^{(1)}$ or $A=A_{2n}^{(2)}$, then
\[  \kappa(\check{\mathring{Q}})=\begin{cases}
Q_l  \quad    A=X_N^{(1)}    \\
\frac{1}{2}Q_l    \quad   A=A_{2n}^{(2)}
\end{cases}.
\]
\el
\bpf
It follows from the discussions in \cite[\S6.5]{Ka}.
\epf

Recall  the torus $\mathring{T}=\Hom(\mathring{P}, \bbC)$.  
Set 
\[ T_{k}:=\{ t\in \mathring{T}   \,|\,      \alpha(t)=1,   \forall \alpha\in (k+\check{ h})M   \} .  \]
 Then $T_k$ is a finite subgroup of $\mathring{T}$. 
An element $t\in \mathring{T}$ is regular if $\alpha(t)\neq 1$ for any $\alpha\in \mathring{\Phi}$.  Equivalently an element $t\in \mathring{T}$ is regular if and only if the stabilizer group of $\mathring{W}$ at $t$ is trivial.   Let $T_k^{\rm reg}$ denote the set of regular elements in $T_k$.

The following is the general Verlinde formula for the trace of $\sigma$ on the space of conformal blocks.
\bt  
\label{Verlinde}
Let $(C,\vec{p})$ be an  $m$-pointed smooth projective curve of genus $g$.
Given a tuple $\vec{\lambda}=(\lambda_1,\lambda_2,\cdots, \lambda_m)$  of elements in $P_k$.    We have the following formula
\be
\label{main_formula}
\tr(\sigma| \dot{V}_{ \vec{\lambda}  }(C) )=|T_k|^{g-1} \sum_{t\in T^{\rm reg}_k/\mathring{W} } \mathring{\chi}_{\lambda_1}(t)\cdots \mathring{\chi}_{\lambda_m}(t) \Delta(t)^{1-g} .
\ee
\et
\bpf
When $A=X_N^{(1)}$,  it is just the usual Verlinde formula(cf.\cite{Be}).  When $A=X_N^{(r)}$ with $r>1$,  it is the main theorem of \cite{Ho2}.  Lemma \ref{M_Untwisted} ensures the data in the  formulas of \cite{Be, Ho2} matches with the note.
\epf

In particular if $C=\bbP^1$  and $\vec{\lambda}=(\lambda,\mu,\nu)$ with $\lambda,\mu,\nu\in P_k$, we have 
\begin{equation}
\label{conformal_block_line}
 \tr(\sigma| \dot{V}_{ \vec{\lambda} }(C)  )=\frac{1}{|T_k|} \sum_{t\in T^{\rm reg}_k/\mathring{W} } \mathring{\chi}_\lambda(t)\mathring{\chi}_\mu(t) \mathring{\chi}_\nu(t)\Delta(t) .  
 \end{equation}

  Let $W_{k}$ denote the affine Weyl group $\mathring{W}\ltimes (k+\check{h} )M$.   We define the dotted action of $W_{k}$  on $\mathring{P}$,
 \[ w\star \lambda=w(\lambda+\mathring{\rho})-\mathring{\rho}, \text{ for any } w\in W_{k}.   \]

  Let $W_{k}^\dagger$  denote the set consisting of the minimal representatives of the left cosets of $\mathring{W}$ in $W_{k}$.  For any $\lambda\in P_k$,  $w\star \lambda\in \mathring{P}^+$ if and only if $w\in  W^\dagger_{k}$.

 The following is a general version of Kac-Walton formula. 
 \bt\label{Kac_Walton}
 For any $\lambda,\mu, \nu\in { P}_k$, we have 
 \[ \tr(\sigma| \dot{V}_{ \lambda,\mu,\nu  }(\bbP^1)  )=\sum_{w\in W_{k}^\dagger }  (-1)^{\ell(w)} \dim \mathring{V}^{\mathring \fg}_{\lambda, \mu, w\star \nu}  .\]
 \et
 \bpf
When $A=X_N^{(1)}$ it is just the usual Kac-Walton formula (cf.\cite[Ex.13.35]{Ka}).  When $A=X_N^{(r)}$ with $r>1$,  it follows from \cite[Theorem 5.11]{Ho2} and Theorem \ref{Jantzen_Thm}.
 \epf

\section{Proofs }

\subsection{Proof of Theorem \ref{Orthnormal_Basis_Thm}}
\subsubsection{Inner product}

\begin{lemma} 
\label{Dual_Coxeter_lemma}
We have the following equality
\[  \check{h}=  
\la \mathring{\rho},  \check{\theta} \ra+1 
\]
\end{lemma}
\begin{proof}
When $A=X_N^{(1)}$,  $\check{\theta}$ is the highest short coroot of $\mathring{\fg}$.  When $A=A_{2n}^{(2)}$,  $\check{\theta}$ is twice of the highest short coroot of $\mathring{\fg}$.  Then it is easy to check that $\check{h}=(\mathring{\rho},  \check{\theta}  )+1$    by comparing the numerical labels in the tables of \cite[\S4.8]{Ka} and Tables (\ref{Table_Dual_Coxeter_1}) (\ref{Table_Dual_Coxeter_2}).

    When $A=X_N^{(r)}$ with $r>1$ and $A\neq A_{2n}^{(2)}$, $\la\mathring{\rho},  \check{\theta} \ra+1$ is the Coxeter number of dual root system of $\mathring{\fg}$.  By comparing the table in \cite[\S6.1]{Ka},  it is also easy to see the equality. 

\end{proof}

To prove the Hermitian form defined in (\ref{Inner_Product}) is an inner product,  it is enough to show that $\Delta(t)>0$ for any $t\in \Sigma_k$.    By Weyl denominator formula 
\[ \Delta(t)=\prod_{\alpha\in \mathring{\Phi} }  (1-\alpha(t)) =  | \prod_{\alpha\in \mathring{\Phi}^+ }  (1-\alpha(t) )    |^2  . \]

When $A=X_N^{(1)}$ or $A=A_{2n}^{(2)}$, for any $\lambda\in P_k$ and $\alpha\in \mathring{\Phi}^+$,  we have $0\leq (  \lambda| \alpha ) \leq  k$.  
By Lemma \ref{Dual_Coxeter_lemma},   $(  \mathring{\rho} | \theta)= \la \mathring{\rho},  \check{\theta} \ra =\check{h}-1$. Since $\theta$ is the highest root of $\mathring{\fg}$,  for any $\alpha\in \mathring{\Phi}^+$ we have
\[0<  (\mathring{\rho}|\alpha )\leq (\mathring{\rho}| \theta )=\check{h}-1 .  \]  
In particular 
\[ 0<(\lambda+\mathring{\rho}|\alpha)\leq  k+\check{h}-1.\]
It shows that $\alpha(t)$ can not be equal to $1$.  Hence $\Delta(t)>0$ for any $t\in \Sigma_k$.

When $A=X_N^{(r)}$ with $r>1$ and $A\neq  A_{2n}^{(2)}$,  for any $\check{\lambda}\in \check{P}_k$, we have $0\leq  \la \alpha,  \check{ \lambda}  \ra \leq  k$. 
By Lemma \ref{Dual_Coxeter_lemma},  for any $\alpha\in \mathring{\Phi}^+$,  we always have
\[  0< \la \alpha,  \check{\mathring \rho}+\check{\lambda} \ra <  k+\check{h}.\]  
Hence in this case,  we also have $\Delta(t)>0$ for any $t\in \Sigma_k$.

\subsubsection{Orthonormal basis}
Recall that $T_k$ is the group consisting of $t\in \mathring{T}$ such that $\alpha(t)=1$ for any $\alpha\in (k+\check{h})M$.   
The group $T_k$ is naturally isomorphic to 
\begin{equation} 
\label{Natural_Iso}
\begin{cases}
\mathring{P}/(k+\check{h})M   \quad  \text{if } A=X_N^{(1)} \text{ or } A=A_{2n}^{(2)}  \\
\check{\mathring{P}}/ (k+\check{h})\check{\mathring{Q}}  \quad  \text{ otherwise}
\end{cases} . \end{equation}

\bl
\label{Cardinality_group_lem}
$|T_k|= |  \mathring{P}/(k+\check{h})M |$.
\el
\bpf
When $A=X_N^{(1)}$ or $A=A_{2n}^{(2)}$,  it simply follows from the natural isomorphism (\ref{Natural_Iso}).  

When $A=X_N^{(r)}$ with $r>1$ and $A\neq A_{2n}^{(2)}$,  $\check{\mathring{P}}/ (k+\check{h})\check{\mathring{Q}}$ is Pontryagin dual to  $\mathring{P}/(k+\check{h})\mathring{Q}$.   Then the lemma follows from the natural isomorphism (\ref{Natural_Iso}).
\epf

We  recall  the finite set $\Sigma_k$ in $\mathring{T}$.  It is clear that  $\Sigma_k$ is a subset of $T^{\rm reg}_k$.     

When $A=X_N^{(1)}$ or $A=A_{2n}^{(2)} $,  let $\eta_1:  P_k\to  \Sigma_k$ be the map by assigning $\lambda\in P_k$ to $t_\lambda: = e^{\frac{2\pi i}{k+\check{h}}( \mathring{\rho}+\lambda|\cdot )} \in \Sigma_k $.  
When $A=X_N^{(r)}$ with $r>1$ and $A\neq A_{2n}^{(2)} $,  let $\eta_2:  \check{P}_k\to \Sigma_k$ be the map assigning $\check{\lambda}\in \check{P}_k$ to $t_{\check{\lambda}}:=   e^{\frac{2\pi i}{k+\check{h}}\la\check{\mathring \rho}+\check{\lambda}, \cdot \ra  } $.

\bl
\label{Embedd_lem}
The maps $\eta_1,\eta_2$ defined above are bijections. 
\el
\bpf
We first assume that  $A=X_N^{(1)}$ or $A=A_{2n}^{(2)}$.    It is enough to show that $\eta_1$ is injective, equivalent it is enough to show that  $i_1: P_k \to \mathring{P}/(k+\check{h})M$ the map  given by 
\[  \lambda\mapsto  \lambda+\mathring{\rho} \mod  (k+\check{h})M   \]
is injective. For any $\lambda\in (k+\check{h})M$ and $w\in \mathring{W}$, we always have 
\begin{equation}
\label{equ_1}
  \la \lambda, w(\check{\theta} )\ra \in (k+\check{h})\bbZ.
  \end{equation}

For any $\lambda_1, \lambda_2\in P_k$, we have  $0\leq \la \lambda_i, \check{\theta} \ra\leq k$.  It also follows that for any $w\in \mathring{W}$,  
\begin{equation}\label{equ_2}  -k \leq  \la \lambda_i, w(\check{\theta}) \ra\leq k   . \end{equation}

If  $i(\lambda_1)=i(\lambda_2)$, equivalently $\lambda_1-\lambda_2\in (k+\check{h})M$,   
then by (\ref{equ_1}) and (\ref{equ_2}) we have  
\[ \la\lambda_1-\lambda_2,  w(\check{\theta})  \ra =0 ,\]
for any $w\in \mathring{W}$.  Since $\{ w(\check{\theta})\,|\, w\in \mathring{W} \}$ spans $\mathring{\fh}$, it follows that $\lambda_1=\lambda_2$.  It proves the injectivity of $i_1$.    

When $A=X_N^{(r)}$ with $r>1$ and $A\neq A_{2n}^{(2)}$, we consider the map $i_2: \check{P}_k \to \check{\mathring{P}}/ (k+\check{h})\check{\mathring Q}$ given by 
\[\check{\lambda}\mapsto  \check{\lambda}+\check{\mathring \rho}  \mod  (k+\check{h}) \check{\mathring{Q}} .\]
In this case the lemma can be proved similarly. 
\epf

\bl 
\label{cardinality_lem}
 When $A=X_N^{(r)}$ with $r>1$ and $A\neq  A_{2n}^{(2)}$, we have
$|P_k|=|\check{P}_k|$.
\el
\bpf
When $A\neq A_{2n-1}^{(2)}$ and $A\neq D_{n+1}^{(2)}$,  it is clear that $|P_k|=|\check{P}_k|$ since the dual Lie algebra of $\mathring{\fg}$ is isomorphic to $\mathring{\fg}$.   

When $A= A_{2n-1}^{(2)}$ or $A= D_{n+1}^{(2)}$,  we define a map $\epsilon: \mathring{P}\to \check{\mathring P}$ given by $\omega_i\mapsto \check{\omega}_{\ell-i}$ where $\ell$ is the rank of $\mathring \fg$.    
Note that $\epsilon$ induces the dual map $\check{\epsilon}:  \mathring{Q}\to  \check{\mathring Q}$.  Observe that $\check{\epsilon}( \theta )=\check{\theta}$, where $\check{\theta}$ is the highest coroot of $\mathring{\fg}$.  
It is now clear that $\epsilon$ defines a bijection $P_k\simeq \check{P}_k$.    Hence the lemma is proved.
\epf

\bco
\label{Sigma_P}
For any generalized Cartan matrix $A$ of affine type, we have  $|\Sigma_k|=|P_k|$.  
\eco
\bpf
It follows from Lemma \ref{Cardinality_group_lem} and Lemma \ref{Embedd_lem}.
\epf

\bl
\label{Fundamental_Set}
$\Sigma_k$ is the fundamental set of $T^{\rm reg}_k$ with respect to the action of the Weyl group $\mathring{W}$, i.e. for any $t\in T^{\rm reg}_k$ there exists a unique element $t_0\in \Sigma_k$ and a unique $w\in \mathring{W}$ such that $w(t_0)=t$.
\el
\bpf
By Lemma \ref{Embedd_lem}, we are reduced to consider the action of the Weyl group $\mathring{W}$ on $\mathring{P}/(k+\check{h})M$. 

When $A=X_N^{(1)}$ or $A=A_{2n}^{(2)}$,   it follows from the fact that the set  $\{\mathring{\rho}+ \lambda \,|\,  \lambda\in P_k   \}$ is exactly those integral weights sitting in the interior of fundamental alcove with respect to the action of the affine Weyl group $\mathring{W}\ltimes  (k+\check{h})M$ on $\mathring{P}\otimes \bbR$.   

When $A=X_N^{(r)}$ with $r>1$ and $A\neq A_{2n}^{(2)}$,   it follows from the fact that the set  $\{ \check{\mathring{\rho}}+ \check{\lambda} \,|\,  \lambda\in \check{P}_k   \}$ is  those integral coweights sitting in the interior of fundamental alcove with respect to the action of the affine Weyl group $\mathring{W}\ltimes  (k+\check{h})\check{\mathring Q}$ on $\check{\mathring{P}}\otimes \bbR$.

\epf

From this lemma, we immediately get $\Sigma_k=T_k^{\rm reg}/\mathring{W}$.  By Theorem \ref{Verlinde}, Lemma \ref{Cardinality_group_lem} and Lemma \ref{Fundamental_Set},  Lemma \ref{Conformal_Orth} is equivalent to 
\begin{equation}\label{equ_char}
  \frac{1}{|  \mathring{P}/(k+\check{h})M   |} \sum_{t\in \Sigma_k}  \mathring{\chi}_\lambda(t) \mathring{\chi}_{\mu^*}(t)\Delta(t) =\delta_{\lambda\mu} , \end{equation}
for any $\lambda, \mu\in P_k$.   Recall the definition of the Hermitian form $(\cdot,\cdot)$ in (\ref{Inner_Product}),  the formula (\ref{equ_char}) implies that $ \{\mathring{\chi}_\lambda\,|\, \lambda\in P_k\}$ is an orthonormal subset of $R_k(A)$.

In the end in view of Corollary \ref{Sigma_P}, the dimension $R_k(A)$ is equal to the cardinality of $P_k$.  It follows that  $ \{\mathring{\chi}_\lambda\,|\, \lambda\in P_k\}$ is indeed an orthonormal basis of $R_k(A)$ with respect to the inner product $(\cdot,\cdot)$.

\subsection{Proof of Theorem \ref{Integrality_Thm}}

From the proof of Theorem \ref{Orthnormal_Basis_Thm}, the Verlinde formula (Theorem \ref{Verlinde}) implies that 
\begin{equation} \label{fusion_trace}
c_{\lambda\mu}^\nu=   \tr(\sigma| \dot{V}_{\lambda, \mu,\nu^*} ^{\dot \fg}) , \end{equation}
for any $\lambda, \mu,\nu\in P_k$.  
By Kac-Walton formula (Theorem \ref{Kac_Walton}) we have 
\be\label{fusion_kac_walton}
 c_{\lambda\mu}^\nu=  \sum_{w\in W_{k}^\dagger }  (-1)^{\ell(w)} \dim  \mathring{V}_{\lambda, \mu, w(\mu^*) }^{\mathring{\fg}}   .\ee

It follows that for any $\lambda, \mu,\nu\in P_k$,  $c_{\lambda\mu}^\nu\in \bbZ$.  

We now proceed to prove the stabilization of fusion coefficients.  We first observe that $\la \lambda^*, \check{\theta} \ra=\la\lambda,  \check{\theta}  \ra $ for any $\lambda\in \mathring{P}$, since $\check{\theta}^*=\check{\theta}$.  Hence the condition $\la\lambda+\mu+\nu, \check{\theta}  \ra\leq  2k$ is equivalent to the condition $\la \lambda+\mu+\nu^*,  \check{\theta}  \ra\leq  2k$. 
 
 When $A=X_N^{(1)}$,  the stabilization follows from \cite[Proposition 4.3]{Be}.

When $A=X_N^{(r)}$ with $r>1$,  $\dot \fg$ is simply-laced, and $\lambda,\mu,\nu$ satisfies the following inequality
\[   \la \iota(\lambda)+\iota(\mu)+\iota(\nu)^*, \check{\dot \theta} \ra\leq 2k   . \]

 Again in view of \cite[Proposition 4.3]{Be}  we have 
\[ \dot{V}_{\lambda, \mu,\nu^*}(\bbP^1)\simeq  (\dot{V}_\lambda\otimes \dot{V}_\mu\otimes \dot{V}_{\nu^*})_{\dot \fg} . \]
 In particular it follows that 
 \[ \tr(\sigma| \dot{V}_{\lambda, \mu,\nu^*}(\bbP^1)  )=\tr(\sigma|   (\dot{V}_\lambda\otimes \dot{V}_\mu\otimes \dot{V}_{\nu^*})_{\dot \fg}   )=\dim  (\mathring{V}_\lambda\otimes \mathring{V}_\mu\otimes \mathring{V}_{\nu^*}) _{\mathring{\fg}} ,  \]
 where the second equality follows from Theorem \ref{Jantzen_Thm}, since the tensor invariant space and tensor co-invariant space of $\dot \fg$ and $\mathring{\fg}$ are naturally isomorphic. 
  
 \subsection{Proof of Theorem \ref{Twining_fusion}}

Let $A_1=A_{2n}^{(2)}$ and $A_2=C_n^{(1)}$.
The underlying simple Lie algebras $\mathring{\fg}_1$ and $\mathring{\fg}_2$ are both isomorphic to the simple Lie algebra of type $C_n$.  We use notation $\mathring{\fg}$ for type $C_n$ simple Lie algebra to identity $\mathring{\fg}_1$ and $\mathring{\fg}_2$.  
Let $M_1$ (resp. $M_2$) be the lattice in $\mathring{\fh}^*$ associated to the generalized Cartan matrix $A_1$ (resp. $A_2$ ) of level $2k+1$ (resp. $k$).  
Then we have
\[ M_1  =(2k+1+ 2n+1) \frac{1}{2}\mathring{Q}_l   =  (k+n+1)\mathring{Q}_l=(k+n+1)\mathring{Q}_l =M_2 . \]

Let $(\cdot |\cdot)_i$ be the  inner product on $\mathring{\fh}$ associated to $A_i$ for each $i=1,2$.  Note that $(\cdot |\cdot)_1=2(\cdot |\cdot )$.  Let $P_{k}(A_i)$ be the finite set of dominant weights associated to $A_i$ as in (\ref{Weight_Level}) for each $i=1,2$.
In particular we have 
\[  P_{2k+1}(A_1)=\{ \lambda\in \mathring{P}   \,|\,   (  \lambda |  \theta  )_1 \leq   2k+1       \}  \]
and 
\[ P_k (A_2)  =\{ \lambda\in  \mathring{P}   \,|\,     (  \lambda |  \theta  )_2   \leq  k    \}   .\]
It is now not difficult to see that 
\[ P_{2k+1}(A_1)= P_k (A_2).  \]
We set $P_k:=P_k (A_2)=P_{2k+1}(A_1)$. Then the fusion ring $R_{2k+1}(A_1)$ and $R_k(A_2)$ have bases indexed by the same set $P_k$.  

Let $\Sigma_k(A_i)$ be the finite set of $\mathring{T}$ defined as in (\ref{Finite_Set_torus}) for each $i=1,2$. 
We observe that 
\[  \Sigma_{2k+1}(A_1)=\Sigma_k(A_2),\]
since for any  $\lambda\in P_k$  we have
\[ e^{\frac{2\pi i}{ 2k+1+2n+1  } ( \lambda+\mathring{\rho} | \cdot  )_1}  =  e^{\frac{2\pi i}{ k+n+1  } ( \lambda+\mathring{\rho} | \cdot  )_2    }.   \]  
From Formula (\ref{Fusion_Coefficient}), we see that the fusion rings $R_{2k+1}(A_1)$ and $R_k(A_2)$ are indeed isomorphic.

\section{Modular S-matrix}
The relation between the fusion ring and modular $S$-matrix is well-known  when the genearlized Cartan matrix is of untwisted affine type, see details in \cite[\S13.8, Ex.13.34,13.35]{Ka}.
In the case of twisted affine case, we use the fusion rings that we have just defined to relate to the $S$-matrix for twisted affine algebras (cf.\cite[\S13.9]{Ka}).   For convenience we still treat all generalized Cartan matrices of affine type generally. 

 We first assume $A=X_N^{(r)}$ with $r>1$ but $A\neq A_{2n}^{(2)}$.  
 For any $A$, we attach the adjacent Cartan matrix $A'$ as follows:
  \begin{equation}
 \label{Table_S_Matrix}
 \begin{tabular}{|c  | c | c |c |c |c | c| c| c|c |c|c|c|c|c |c ||} 
 \hline
$ A$    &  $A_{2n-1}^{(2)}$  & $D_{n+1}^{(2)}$  & $ E_6^{(2)}$  &  $D_4^{(3)} $    \\ [0.5ex] 
 \hline
$A'$    &   $D_{n+1}^{(2)}$  &  $A_{2n+1}^{(2)}$  &  $E_6^{(2)}$ &  $D_4^{(3)}$  \\ 
 \hline
\end{tabular}.
\end{equation}

 We define a linear isomorphism $\epsilon:  \mathring{\fh}'^* \simeq \mathring{\fh}$ such that  
 \begin{enumerate}
\item If $A=A_{2n-1}^{(2)}$ or $A=D^{(2)}_{n+1}$, then
 $\epsilon(\omega'_i)=\check{\omega}_{n-i}$ for any $i=1,\cdots, n$ ; 
 \item otherwise $\epsilon(\omega'_i)=\check{\omega}_i$ for any $i=1,\cdots, n$,
 \end{enumerate}
 where $\check{\omega}_i$ denotes the fundamental coweight of $\mathring{\fg}$.  
 
 Let $\check{\theta}'$ be an element in the coroot lattice of $\mathring{\fg}'$ defined similarly as $\check{\theta}$ in (\ref{Dual_Theta}), and  let $P'_k$ be similarly defined as $P_k$ in (\ref{Weight_Level}).   
  \bl
  The map $\epsilon$ restricts to a bijection $\epsilon:  P'_k\simeq \check{P}_k$.
 \el
 \begin{proof}
 The root system of  $\mathring{\fg}'$ is dual to $\mathring{\fg}$.   The lemma is clear by looking at the Dynkin diagram of $\mathring{\fg}$ and $\mathring{\fg}'$.  
 \end{proof}

 We now identify $\mathring{\fh}$ and $\mathring{\fh}'$ via the isomorphism $\varsigma:=\epsilon^*\circ \kappa$ from $ \mathring{\fh}$ to $\mathring{\fh}'$, where $\epsilon^*$ is the dual operator of $\epsilon: \mathring{\fh}'^*\simeq \mathring{\fh}$. 
  
 When $A=X_N^{(1)}$ or $A=A_{2n}^{(2)}$,  we take $A'=A$ and take $\varsigma$ to be the identity on $\mathring{\fh}$.  
 
 When $A=X_N^{(r)}$ with $r>1$ and $A\neq A_{2n}^{(2)}$,  for any $\mu'\in P_k$, let $t_{\mu'}$ denote the element in $\Sigma_k$ associated to $\epsilon(\mu')\in \check{P}_k$.   When $A=X_N^{(1)}$ or $A=A_{2n}^{(2)}$, for any $\mu'\in P'_k=P_k$, we denote by $t_{\mu'}$ the element in $\Sigma$ associated to $\mu'\in P_k$.  Put 
 \begin{equation}
 \label{S_Matrix}
 S_{\lambda,\mu'}=i^{|\mathring{\Phi}^+|}|\mathring{P}/(k+\check{h})M|^{-1/2} J_{\lambda}( t_{\mu' }) ,    \end{equation}
 where $i$ is the imaginary unit. 
 
 The matrix $S=(S_{\lambda,\mu'})$ is indexed by $P_k\times P'_k$.  We call it the modular matrix from $A$ to $A'$.  
  Let $S^t$ (resp. $\bar{S}^t$) be the transpose (resp. conjugate transpose) of $S$.    Then $S^t$ (resp. $\bar{S}^t$) is a matrix indexed by $P'_k\times P_k$.
  
  \bl
 $S\bar{S}^t$ is the identify matrix indexed by $P_k \times P_k$.  
  \el
  \bpf
 Note that the number of positive roots of type $B_n$ is equal to the number of positive roots of type $C_n$ (cf.\cite[\S12.2, Table 1]{Hu}). It is easy to see that $i^{|\mathring{\Phi}^+|}=i^{|\mathring{\Phi}'^+ | }$ for any $A$, where $\mathring{\Phi}'^+$ is the set of positive roots of $\mathring{\fg}'$.  
  Hence the lemma is  equivalent to that $\{\mathring{\chi}_\lambda \,|\, \lambda\in P_k\}$ is an orthonormal basis of $R_k(A)$ with respect to the inner product $(\cdot,\cdot)$ defined in (\ref{Inner_Product}).  
  \epf

  Note that $(A')'=A$. Let $S'$ be the modular matrix $(S_{\lambda',\mu})$  from $A'$ to $A$.
 \bl
 $S^t=S'$.
 \el
 \bpf
 When $A=X_N^{{(1)}}$ or $A=A_{2n}^{(2)}$,  the lemma is clear since $J_{\lambda}(t_\mu)=J_{\mu}(t_\lambda)$.
 
We now assume $A=X_N^{(r)}$ with $r>1$ and $A\neq A_{2n}^{(2)}$.  Note that $M=\mathring{Q}$ (resp. $M'=\mathring{Q}'$).
moreover 
\[  \check{h}'=\check{h}, \quad \epsilon(\mathring{P}')=\check{\mathring{P}}, \text{ and } \epsilon(\mathring{Q}')=\check{\mathring{Q}}  .\]
It follows that $\epsilon$ induces the isomorphism
\[  \mathring{P}'/(k+\check{h}' )\mathring{Q}' \simeq  \check{\mathring{P}}/(k+\check{h})\check{\mathring{Q}}  . \]
From (\ref{J_function}), we observe that  $J_\lambda(t_{\mu'})=J_{\mu'}(t_\lambda)$ for any $\lambda\in P_k$ and $\mu'\in P'_k$.  It concludes the Lemma.
 \epf

  For any dominant integral weight $\Lambda$ of $\fg(A)$ of level $k$, we denote by $\CH_\Lambda$ the irreducible integrable representation of highest weight $\Lambda$.   Let $\chi_\Lambda$ denote the normalized character of $\fg(A)$, i.e. 
$\chi_\Lambda= e^{-m_\Lambda \delta} {\rm ch}_{\CH_\Lambda}$, where ${\rm ch}_{\CH_\Lambda}$ is the character of $\CH_\Lambda$ and the number $m_\Lambda$ is the so-called modular anomaly (cf.\cite[\S12.7.5]{Ka}).  
Let $V(A,k)$ be the vector space spanned by the normalized characters  $\chi_\Lambda$ of level $k$.   Then the matrix $S$ gives a linear transformation from $V(A,k)$ to $V(A',k)$.  For precise transformation formula, see \cite[Theorem 13.8, 13.9]{Ka}. In the literature the formula (\ref{S_Matrix}) is called Kac-Peterson formula  (cf.\cite{Ka,Ka1}) for the modular $S$-matrix as the linear transformation from $V(A,k)$ to $V(A',k)$.


 \section*{Acknoledgements}  
 This note originated from the comments of I.\,Cherednik on the work \cite{Ho2}. The elementary definition of fusion rings presented in this note should be well-known in the case of untwisted affine Lie algebra, but the author is not able to find an appropriate reference.  
This work is partially supported by the Simons Foundation collaboration grant 524406.

 \end{document}